\documentclass[twoside]{wujart}
\usepackage{psfig}
\usepackage{etex}

\usepackage{amsthm}
\usepackage{mathtools}
\usepackage{amssymb}
\usepackage{comment}
\usepackage{enumitem}
\usepackage{hyperref}
\usepackage[percent]{overpic}
\usepackage[english]{babel}
\usepackage{microtype}
\usepackage[sort,nocompress]{cite} 

\usepackage[format=hang,labelfont=bf,tablename=Tab]{caption}

\usepackage{tocloft}
\usepackage{titlesec}
\titleformat{\section}
  {\normalfont\bfseries}{\thesection.}{0.5em}{}
\titleformat{\subsection}
  {\normalfont\bfseries}{\thesubsection.}{0.5em}{}
 
\renewcommand\thesection{\arabic{section}}
\renewcommand\thesubsection{\thesection.\arabic{subsection}}

\newcommand{\vectornorm}[1]{\left|\left|#1\right|\right|}
\newcommand{\rr}[0]{\mathbb{R}}
\newcommand{\zz}[0]{\mathbb{Z}}
\newcommand{\nn}[0]{\mathbb{N}}

\DeclareMathOperator{\inter}{int}

\DeclareMathOperator{\dom}{dom}

\newtheoremstyle{mytheoremstyle} 
    {\topsep}                    
    {\topsep}                    
    {\itshape}                   
    {}                           
    {\scshape}                   
    {.}                          
    {.5em}                       
    {}  

\theoremstyle{mytheoremstyle}

\newtheorem{thm}{Theorem}
\newtheorem{cor}[thm]{Corollary}
\newtheorem{lem}[thm]{Lemma}

\newtheorem{defn}[thm]{Definition}

\begin{document}


\author{Aleksander Czechowski$^1$ and Piotr Zgliczy\'nski$^1$ \\ 
  {\small \rm $^1$Institute of Computer Science and Computational Mathematics\\ Jagiellonian University\\ \L ojasiewicza 6, Krak\'ow, 30-348 Poland\\
    e-mail: {\it czechows@ii.uj.edu.pl, zgliczyn@ii.uj.edu.pl} } }

\renewcommand*{\thefootnote}{\fnsymbol{footnote}}

\title{Rigorous numerics for PDEs with indefinite tail: \linebreak
  existence of a periodic solution of the Boussinesq equation with time-dependent forcing\footnote{AC was supported by the Foundation for Polish Science under the MPD Programme 
    ‘‘Geometry and Topology in Physical Models’’,
  co-financed by the EU European Regional Development Fund, Operational Program Innovative Economy 2007-2013. PZ was supported by Polish
National Science Centre grant 2011/03B/ST1/04780.}}

\maketitle
\renewcommand*{\thefootnote}{\arabic{footnote}}

\setcounter{footnote}{0}

\abstract{We consider the Boussinesq PDE perturbed by a time-dependent forcing.
Even though there is no smoothing effect for arbitrary smooth initial data,
we are able to apply the method of self-consistent bounds to deduce the existence of smooth classical periodic solutions 
in the vicinity of 0.
The proof is non-perturbative and relies on construction of periodic isolating segments
in the Galerkin projections.}

\keywords{Boussinesq equation, ill-posed PDEs, periodic solutions, isolating segments.}


\section{Introduction}

In recent years pioneering contributions have been made
in the field of rigorous computer-assisted results for dynamics of dissipative PDEs~\cite{ZgliczynskiMischaikow, Zgliczynski2, Zgliczynski3, Zgliczynski4, ZAKS,
Cyranka, CyrankaZgliczynski, ArioliKoch1, ArioliKoch2, Day, Gameiro}.
The methods exploit the smoothing property of the system to apply either topological
or functional-analytic tools.
However, little attention has been paid to apply these methods to other types of evolution PDEs, such as the
ones with tail of saddle type. In such problems we need to deal with an infinite number of strongly repelling and strongly attracting directions.
From the point of view of topological methods this situation is just as good as the dissipative one
-- for example Theorem~2.3 in~\cite{Zgliczynski4} is formulated in a way that is readily applicable
to finding equilibria of these systems.
Our goal in this paper is to take this approach one step further and use 
a topological tool of periodic isolating segments to prove the existence of periodic in time solutions
in a nonautonomously perturbed equation of such type.

Our example will be the Boussinesq equation~\cite{Boussinesq} but without much difficulty the methods can be applied
to produce similar results in other systems with an indefinite tail of saddle type.

\subsection{The forced Boussinesq equation}

We consider the following second order nonlinear equation perturbed by a time-dependent forcing term:
\begin{equation}\label{eq:bsq}
  u_{tt} = u_{xx} + \beta u_{xxxx} + \sigma (u^{2})_{xx} + \epsilon f(t,x)\,.
\end{equation}
On $u$ and $f$ we impose periodic and even boundary conditions and a zero-average condition in $x$:
\begin{align}
  u(t,x+2\pi) &= u(t,x)\,,\label{eq:per} \\
  u(t,-x) &= u(t,x)\,,\label{eq:ev} \\
  \int_0^{2\pi} u(t,x) dx &= 0\,,\label{eq:av} \\
  f(t,x+2\pi) &= f(t,x)\,,\label{eq:fper} \\
  f(t,-x) &= f(t,x)\,,\label{eq:fev} \\
  \int_0^{2\pi} f(t,x) dx &= 0\,.\label{eq:fav}
\end{align}

For $\beta > 0$ the unperturbed equation
\begin{equation}\label{eq:bsqAut}
  u_{tt} = u_{xx} + \beta u_{xxxx} + \sigma (u^{2})_{xx}
\end{equation}
is the ``bad'' Boussinesq equation and was derived by Boussinesq~\cite{Boussinesq} as a model for shallow water waves.
The equation is famous for its ill-posedness.
Indeed, when looking at its linear part
\begin{equation}\label{eq:bsqLin}
  u_{tt} = u_{xx} + \beta u_{xxxx}
\end{equation}
one can observe a rapid growth in high Fourier modes for almost all initial data,
hence a consequent loss of regularity of the solution.
This is a significant complication in the numerical analysis of~\eqref{eq:bsqAut},
since slightest perturbations of the initial problem can produce a totally different behaviour at output.
Because of that, regularized versions of the equation were considered in numerical studies~\cite{Manoranjan}.
Solutions to the equation~\eqref{eq:bsqAut} were also obtained analytically~\cite{Hirota} and by the inverse scattering method~\cite{Zakharov}.
Our approach is different; we analyze the direction of the vector field on certain subsets of the phase space,
and by a topological method we deduce the existence of smooth, periodic solutions.

Here is an example result illustrating our method.
We define families of functions
\begin{equation}
  \begin{aligned}
  \mathcal{F}_{\tau}^{A} = \{ f:\ f(t,x) = 2 f_1(t) \cos x,\\ f_1 \text{ continuous and $\tau$-periodic}, \ |f_1(t)| \leq 1 \ \forall t \}\,, \\
  \mathcal{F}_{\tau}^{B} = \{ f:\ f(t,x) =  2 \sum_{k=1}^{4} f_k(t) \cos kx,\\ f_k \text{ continuous and $\tau$-periodic}, \ |f_k(t)| \leq 1 \ \forall k,t \}\,.
\end{aligned}
\end{equation}
Then the following theorem holds.
\begin{thm}\label{thm:main}
  For $\sigma=3$, for all $\tau>0$ and all $f(t,x) \in \mathcal{F}_{\tau}^{A} \cup \mathcal{F}_{\tau}^{B}$,
  and for values of $\beta$ and $\epsilon$ given in Tables~\ref{tab1} and~\ref{tab2} there exists a classical $\tau$-periodic in time solution to~\eqref{eq:bsq},
  subject to conditions~\eqref{eq:per}, \eqref{eq:ev} and~\eqref{eq:av}. The solution exists in the vicinity of 0 and the
  bounds on its $L^2$ and $C^0$ norms and the norms of its time derivative are given in Tables~\ref{tab1} and~\ref{tab2}.\footnote{By bounds on $L^2$ and $C^0$ norms of a function $u=u(t,x)$
    we mean upper bounds on $\sup_{t \in \rr} \vectornorm{u(t,\cdot)}_{L^2}$ and $\sup_{t \in \rr} \vectornorm{u(t,\cdot)}_{C^0}$, respectively.}
    The solution and its time derivative are $C^4$ and $C^2$ smooth in $x$, respectively.
\end{thm}

\begin{table}
  \begin{center}
    \caption{Bounds on $\epsilon$ and on the norms of periodic solutions and their time derivatives for $f \in \mathcal{F}_{\tau}^{A}$}
    \scalebox{0.97}{\begin{tabular}{ | c | c | c | c | c | c |}
        \hline
          $\beta$ & $\epsilon$ & $\vectornorm{u(t,\cdot)}_{L^2} \leq$ & $\vectornorm{u(t,\cdot)}_{C^0} \leq$ &
          $\vectornorm{u_t(t,\cdot)}_{L^2} \leq$ &  $\vectornorm{u_t(t,\cdot)}_{C^0} \leq$  \\ \hline
          $1.5$ & $[-0.05,0.05]$ & $0.28862115$ & $0.12440683$ & $0.24610779$ & $0.2143523$ \\ \hline
          $1.75$ & $[-0.1,0.1]$ & $0.41504192$ & $0.1820825$ & $0.39340084$ & $0.42461205$ \\ \hline
          $2.5$ & $[-0.3,0.3]$ & $0.84724825$ & $0.38676747$ & $0.96839709$ & $1.4795576$ \\ \hline
        \end{tabular}}
    \label{tab1}
  \end{center}
\end{table}
\begin{table}
  \begin{center}
    \caption{Bounds on $\epsilon$ and on the norms of periodic solutions and their time derivatives for $f \in \mathcal{F}_{\tau}^{B}$}
    \scalebox{0.97}{\begin{tabular}{ | c | c | c | c | c | c |}
        \hline
          $\beta$ & $\epsilon$ & $\vectornorm{u(t,\cdot)}_{L^2} \leq$ & $\vectornorm{u(t,\cdot)}_{C^0} \leq$ &
          $\vectornorm{u_t(t,\cdot)}_{L^2} \leq$ &  $\vectornorm{u_t(t,\cdot)}_{C^0} \leq$  \\ \hline
          $1.5$ & $[-0.05,0.05]$ & $0.29831987$ & $0.13194161$ & $0.25703095$ & $0.24099758$ \\ \hline
          $1.75$ & $[-0.1,0.1]$ & $0.43198386$ & $0.19524766$ & $0.41478653$ & $0.47720834$ \\ \hline
          $2.5$ & $[-0.3,0.3]$ & $0.88825406$ & $0.41784158$ & $1.0309512$ & $1.637095$ \\ \hline
        \end{tabular}}
    \label{tab2}
  \end{center}
\end{table}

Observe, that $0$ is a constant in time solution of the unperturbed system~\eqref{eq:bsqAut},
hence the requested $\tau$-periodic solution for $\varepsilon=0$. Nevertheless, the method is not perturbative.
We consider a perturbation problem only because it gives a convenient approximation of the periodic solution for $|\epsilon| \neq 0$ small.

The proof is computer-assisted, that means certain inequalities contained in it are verified rigorously by a computer program in interval arithmetics.
The program source code is available at~\cite{Czechowski}.
From Table~\ref{tab3} and equation~\eqref{eq:changeVar} one can also extract the exact bounds on the Fourier coefficients of the solutions, which we do not give here.
By following the steps of the proof it will become clear that we can easily produce results of the same type
for:
\begin{itemize}
  \item any parameters $\sigma \in \rr$ and $\beta > 1$,
  \item any given smoothness $s>5$,
  \item any forcing of the form $f(t,x) = \sum_{k=1}^{n} f_k (t) \cos kx$ where each $f_k$ is continuous and $\tau$-periodic.
\end{itemize}
The periodic solution will
exist for $\varepsilon \in [-\varepsilon_0,\varepsilon_0]$, $\varepsilon_0$ small enough, and we can attempt to verify an explicit range and obtain a bound for the norm
with help of the program.

Let us finish this section with listing some generalizations.
Most of them can be adapted from the previous treatment of dissipative PDEs~\cite{ZgliczynskiMischaikow, Cyranka, CyrankaZgliczynski}
and several would require only little effort to be introduced in this paper.
However, we consider this exposition as a preview of the method. We tried to focus on the key matter,
which is how to deal with the linear instability of the high Fourier modes in a simplest scenario.

\begin{enumerate}[label={\arabic*.},leftmargin=*]
  \item We could take $\beta$ from the range $(0,1]$. Then, the linearized equation~\ref{eq:bsqLin} possesses purely imaginary eigenvalues
  in its low modes. This would involve conducting a finite-dimensional analysis of the higher order terms of the low modes.

  \item We could allow non-zero averages, non-even functions or periodic solutions obtained
  in the proximity of non-zero equilibria of the unperturbed system.

  \item For the forcing term it would have been enough to assume a sufficiently fast decay in the high Fourier terms.

  \item We could consider non-periodic forcings and attempt to prove existence of (not necessarily periodic) solutions that exist for all $t \in \rr$,
  by techniques from~\cite{CyrankaZgliczynski}.
\end{enumerate}
The apparently difficult problems are
\begin{itemize}
  \item proving the existence of periodic solutions which are not obtained as perturbations of stationary points,
  \item proving dynamics more complicated than a periodic orbit (e.g. chaotic dynamics),
  \item proving the existence of periodic orbits in autonomous ill-posed systems.
\end{itemize}
We think that to efficiently treat these cases within the framework of self-consistent bounds
we would need a rigorous integration procedure, akin to the rigorous integration of dissipative PDEs~\cite{Zgliczynski2, Zgliczynski3, Cyranka}.
Obviously, the ill-posedness is a significant issue
and it seems that the integration should be combined with an automatic segment placement in the expanding coordinates.
We are currently looking into the feasability of this approach.

This paper is organized as follows. In Section~\ref{sec:bounds} we invoke the general method of self-consistent bounds
and a result which states that a sequence of solutions for the Galerkin projections converges to a solution of the PDE.
In Section~\ref{sec:segments} we present a result of Srzednicki~\cite{Srzednicki} stating the conditions under which
non-autonomous time-periodic ordinary differential equations -- in our case the Galerkin projections -- have periodic solutions.
In Section~\ref{sec:boussinesq} we apply these tools to the Boussinesq equation~\eqref{eq:bsq}
and prove Theorem~\ref{thm:main}.

\section{The general method of self-consistent bounds}\label{sec:bounds}

In this section we recall the general method of self-consistent bounds, as introduced in the series
of papers~\cite{ZgliczynskiMischaikow, Zgliczynski2, Zgliczynski3}. 
We follow the exposition given in~\cite{CyrankaZgliczynski} for time-dependent systems.

Let $J \subset \rr$ be a (possibly unbounded) interval. We consider a nonautonomous
evolution equation on a real Hilbert space $H$ ($L^2$ in our case) of the form
\begin{equation}\label{eq:abstractPDE}
  \frac{da}{dt} = F(t,a)\,.
\end{equation}

We assume that the set of $x$ such that $F(t,x)$ is defined for every $t \in J$ is dense in $H$ and we
denote it by $\tilde{H}$. By a solution of~\eqref{eq:abstractPDE} we mean a function
$a:J' \to \tilde{H}$, such that $J'$ is a subinterval of $J$, $u$ is differentiable and~\eqref{eq:abstractPDE} is satisfied for all $t \in J'$.

Let $I \subset \zz^d$ and let $H_{k} \subset H$ be a sequence of subspaces with $\dim H_{k} \leq d_1 < \infty$, such that
\begin{equation}\label{eq:prod}
  H = \overline{ \bigoplus_{k \in I} H_k }
\end{equation}
and $H_{k}$'s are pairwise orthogonal. We will denote the orthogonal projection onto $H_k$ by $A_k$ and 
write
\begin{equation}
  a_k := A_k a\,.
\end{equation}
From~\eqref{eq:prod} it follows that $a= \sum_{k \in I} a_k$.

From now on we will fix some (arbitrary) norm $|\cdot|$ on $\zz^d$.
For $n>0$ we set
\begin{equation}
  \begin{aligned}
    X_n &:= \bigoplus_{k \in I,|k|\leq n} H_k\,, \\
    Y_n &:= X_n^{\perp}\,.
  \end{aligned}
\end{equation}
We will denote the orthogonal projections onto $X_n$ and $Y_n$ by $P_n : H \to X_n$ and $Q_n: H \to Y_n$.

\begin{defn}
  Let $J \subset \rr$ be an interval. We say that $F: J \times H \supset \dom{F} \to H$ is admissible, 
  if the following conditions hold for each $i \in \zz^d$ such that $\dim X_i > 0$:
  \begin{itemize}
    \item $J \times X_i \subset \dom{F}$,
    \item $(P_i \circ F)|_{J \times X_i}: J \times X_i \to X_i$ is a $C^1$ function.
  \end{itemize}
\end{defn}

\begin{defn}
  Let $F: J \times \tilde{H} \to H$ be admissible. 
  The ordinary differential equation
  \begin{equation}
    \frac{dp}{dt} = P_n F(t,p), \quad p \in X_n\,,
  \end{equation}
  will be called the $n$-th Galerkin projection of~\eqref{eq:abstractPDE}.
\end{defn}

\begin{defn}
  Assume $F: J \times \tilde{H} \to H$ is an admissible function. Let $m,M \in \nn$ with $m \leq M$.
  A compact set consisting of $W \subset X_m$ and a sequence of compact sets $\{ B_k \}_{k \in I,|k|>m}$ such that $B_k \subset H_k$ 
  form self-consistent bounds if the following conditions are satisfied:
  \begin{itemize}
    \item[\textbf{C1}] For $|k|>M$, $k \in I$ it holds that $0 \in B_k$,
    \item[\textbf{C2}] Let $|k|>m$ and $\hat{a}_k := \max_{a \in B_k} \vectornorm{a}$. Then $\sum_{|k|>m, k \in I} \hat{a}_k^2 < \infty$.
      In particular we have
      \begin{equation}
         W \oplus \prod_{|k|>m} B_k \subset H\,.
      \end{equation}
    \item[\textbf{C3}] The function $(t,u) \to F(t,u)$ is continuous on $J \times W \oplus \prod_{|k|>m} B_k \subset \rr \times H$.
      Moreover, if we define $\hat{f}_k := \sup_{(t,u) \in J \times W \oplus \prod_{k \in I, |k|>m} B_k} \vectornorm{A_k F(t,u)}$, then $\sum_{|k|>m, k \in I} \hat{f}_k^2 < \infty$.
  \end{itemize}
\end{defn}

Given self-consistent bounds formed by $W$ and $\{B_k\}_{k \in I, |k| > m}$, by $T$ (\emph{the tail}) we will denote the set
\begin{equation}
  T:= \prod_{k \in I, |k| > m} B_k \subset Y_m\,.
\end{equation}

The following theorem is a straightforward adaptation of Lemma 5 from~\cite{Zgliczynski2} (see also Section 4 in~\cite{ZgliczynskiNS}) to a nonautonomous setting.

\begin{thm}\label{thm:conv}
  Let $W$ and $\{B_k\}_{k \in I, |k|>m}$ form self-consistent bounds and let \linebreak $\{n_k\}_{k \in \nn} \subset \nn$ be a sequence,
  such that $\lim_{k \to \infty} n_k = \infty$. Assume that for all $k>0$ there 
  exists a solution $x_k: [t_1,t_2] \to W \oplus T$ of
  \begin{equation}
    \frac{dp}{dt} = P_{n_k}\left(F (t,p(t))\right), \quad p(t) \in X_{n_k}\,. 
  \end{equation}
  Then there exists a convergent subsequence $\lim_{l \to \infty} p_{k_l} = p^*$, where $p^*: [t_1, t_2] \to W \oplus T$ is a solution of~\eqref{eq:abstractPDE}.
  Moreover, the convergence is uniform with respect to $t$ on $[t_1, t_2]$.
\end{thm}

It turns out that it is fairly simple to find self-consistent bounds.
In the treatment of evolution PDEs such as Kuramoto-Sivashinsky or Navier-Stokes 
it is enough to take tails of the form $B_k=\{ a \in H_k: \vectornorm{a} \leq C/|k|^s \}$ for $s$ large enough.
This will also be the case in our study of the Boussinesq equation.

\section{Periodic isolating segments}\label{sec:segments}

The purpose of this section is to recall a result of Srzednicki~\cite{Srzednicki} on the existence of periodic orbits
in non-autonomous time-periodic ODEs. For the Boussinesq equation this theorem will be used to treat each of the Galerkin projections of the system.

We consider an ODE
\begin{equation}\label{eq:ODE}
  \dot{x} = g(t,x)\,,
\end{equation}
where $g:\rr \times \rr^n \to \rr^n$ is of class $C^1$ and $\tau$-periodic in $t$.
Let $L^i: \rr^n \to \rr,\ i=1,\dots, k$ be $C^1$ functions and let $r \in \{ 0, \dots, k \}$ be fixed.
We define sets
\begin{equation}
  \begin{aligned}
    S_0&:= \{ x \in \rr^n: L^i (x) \leq 0 \ \forall i=1,\dots ,k \}\,, \\
    S^-_0&:= \{ x \in \rr^n: \exists i \in 1,\dots,r: L^i(x) = 0 \}\,,
  \end{aligned}
\end{equation}
and put $S:=[0,\tau] \times S_0$, $S^-:= [0,\tau] \times S^-_0$.

\begin{defn}
  We call $S$ a periodic isolating segment\footnote{The original paper~\cite{Srzednicki} uses the notion of $(p,q)$-blocks, 
  periodic isolating segments are introduced in later studies~\cite{Srzednicki2, SrzednickiWojcik} as a more general tool.} 
  over $[0,\tau]$ and $S^-$ its exit set iff the following conditions hold
  \begin{itemize}
    \item[(S1)] $g(t,x) \cdot \nabla L^i(x) > 0$ for $t \in [0,\tau]$, $i \in \{1,\dots, r\}$ and $x \in S_0: L^i(x) = 0$.
    \item[(S2)] $g(t,x) \cdot \nabla L^i(x) < 0$ for $t \in [0,\tau]$, $i \in \{r+1, \dots, k\}$ and $x \in S_0: L^i(x) = 0 $.
  \end{itemize}
\end{defn}

\begin{thm}[Theorem 2 in~\cite{Srzednicki}]\label{thm:srzednicki}
  Let $S$ be a periodic isolating segment over $[0,\tau]$ for $g$ defined as above. If $S$ and $S^-$ are compact absolute neighborhood retracts 
  and the difference of their Euler characteristics $\chi(S_0) - \chi(S_0^-)$ is non-zero,
  then there exists a point $x_0 \in \inter S_0$ such that the solution $x(t)$ of~\eqref{eq:ODE} satisfying $x(0)=x_0$ is $\tau$-periodic in $t$ and $x(t) \in \inter S_0$ for all $t \in \rr$.
\end{thm}

In what is below we denote the $i$-th coordinate of a vector $x \in \rr^n$ by $x_i$.

\begin{cor}\label{cor:isegment}
  Let $S_0 := [a_1,b_1] \times \dots \times [a_n,b_n]$ and $g=(g_1,\dots,g_n)$ be defined as above.
  Suppose that for each $i$ it holds that
  \begin{equation}\label{eq:isolation}
    g_i(t,x) g_i(t,y) < 0
  \end{equation}
  for all $x,y \in S_0: x_i = a_i,\ y_i=b_i$ and for all $t \in [0,\tau]$.
  Then there exists a $\tau$-periodic in $t$ solution of~\eqref{eq:ODE} with values contained in $\inter S_0$.
\end{cor}

\begin{proof}
  After rearranging the coordinates we can assume that
  \begin{equation}\label{eq:rearr}
    \begin{aligned}
      g_i(t,x) > 0, \ &\text{ for } x \in S_0:\ x_i \in \{ b_1, \dots, b_r, a_{r+1}, \dots, a_n \}\,, \\
      g_i(t,x) < 0, \ &\text{ for } x \in S_0:\ x_i \in \{ a_1, \dots, a_r, b_{r+1}, \dots, b_n \}\,.
    \end{aligned}
  \end{equation}
  for all $t \in [0,\tau]$ and some $r \in \{1,\dots,n\}$.

  Let $L^i(x) := (x_i - a_i)(x_i - b_i)$, $i = 1,\dots,n$. Then $S_0$ is given by the sequence $\{L^i\}$
  and 
  \begin{equation}
    \nabla L^i(x) = 2 x_i - a_i - b_i\,.
  \end{equation}
  Now conditions (S1) and (S2) immediately follow from~\eqref{eq:rearr}, hence $S=[0,\tau] \times S_0$ forms a periodic isolating segment.
  Since $S_0^-$ consists of $r$ opposite faces, we have $\chi(S_0)-\chi(S_0^-) = (-1)^r$ and the assertion of Theorem~\ref{thm:srzednicki} follows. 
\end{proof}
Later on we will refer to condition~\eqref{eq:isolation} as the \emph{isolation conditions} or the \emph{isolation inequalities}.

\section{Application to the Boussinesq PDE}\label{sec:boussinesq}

\subsection{The Boussinesq equation in the Fourier basis}

Our goal in this section is to express our PDE in coordinates suitable for
application of Corollary~\ref{cor:isegment} to the Galerkin projections.
For that purpose we express the problem in the Fourier basis and diagonalize its linear part.

By formally substituting $u(t,x)= \sum_{k \in \zz} u_{k}(t) e^{ikx}$ into the Boussinesq equation~\eqref{eq:bsq} we obtain
an infinite ladder of second order equations
\begin{equation}
  \ddot u_{k} = k^{2} ( \beta k^{2} - 1 ) u_{k} - \sigma k^2 \sum_{k_1 \in \zz} u_{k_1}u_{k-k_1} + f_{k}(t),\ k \in \zz\,.
\end{equation}
Since $u$ is real and even in $x$, we have $u_{k}=u_{-k}$ and $u_{k} \in \rr$ for all $k \in \zz$.
Moreover, from~\eqref{eq:av} we have $u_0=0$. After these substitutions and rewriting the system
as a first order system we obtain the following equations
\begin{equation}\label{eq:infiniteOdeWF}
  \begin{aligned}
    \dot u_{k} &= v_k\,, \\
    \dot v_k &= k^{2} ( \beta k^{2} - 1 ) u_{k} - 2 \sigma k^2 \sum_{k_1 \geq  1} u_{k_1+k}u_{k_1} - \sigma k^2 \sum^{k-1}_{k_1 = 1} u_{k_1}u_{k-k_1} + \epsilon f_{k}(t),\ k \in \nn^+\,.
  \end{aligned}
\end{equation}

As one can see, the linear part of~\eqref{eq:infiniteOdeWF} is already in a block-diagonal form. All we need is to diagonalize each of the blocks.
From now on we assume that $\beta>1$.
After a simple calculation we see that the eigenvalues of the linear part of~\eqref{eq:infiniteOdeWF}
are $\pm \sqrt{ k^2 (\beta k^2 - 1 ) }$ with eigenvectors $\left[1, \pm \sqrt{ k^2 (\beta k^2 - 1 ) } \right]^T$, respectively.
We introduce the variables $u_k^+$ and $u_k^-$ such that
\begin{equation}\label{eq:changeVar}
  \left[ \begin{array}{c} u_k \\ v_k \end{array} \right] =
    \left[ \begin{array}{cc} 1 & 1 \\ \sqrt{ k^2 (\beta k^2 - 1 ) } & -\sqrt{ k^2 (\beta k^2 - 1 ) } \end{array} \right] \left[ \begin{array}{c} u_k^+ \\ u_k^- \end{array} \right]\,.
\end{equation}
We have
\begin{equation}
  \left[ \begin{array}{c} u_k^+ \\ u_k^- \end{array} \right] =
    \left[ \begin{array}{cc} \frac{1}{2} & \frac{1}{2 \sqrt{ k^2 (\beta k^2 - 1 ) }} \\ \frac{1}{2} & -\frac{1}{2\sqrt{ k^2 (\beta k^2 - 1 ) }} \end{array} \right]
      \left[ \begin{array}{c} u_k \\ v_k \end{array} \right]\,,
\end{equation}
and our equations become
\begin{equation}\label{eq:diag}
  \begin{aligned}
    \dot u_k^+ &= \sqrt{ k^2 (\beta k^2 - 1 ) }u_k^+  + \frac{ \sigma k^2 N_k(u) + \epsilon f_k(t)}{2 \sqrt{ k^2 (\beta k^2 - 1 ) }}, \\
    \dot u_k^- &= -\sqrt{ k^2 (\beta k^2 - 1 ) }u_k^- - \frac{ \sigma k^2 N_k(u) + \epsilon f_k(t)}{2 \sqrt{ k^2 (\beta k^2 - 1 ) }}, \ k \in \nn^+ \,,
  \end{aligned}
\end{equation}
where
\begin{equation}
  N_{k}(u) := -2 \sum_{k_1 \geq  1} u_{k_1+k}u_{k_1} - \sum^{k-1}_{k_1 = 1} u_{k_1}u_{k-k_1}\,.
\end{equation}
We remark that it is unprofitable to rewrite the convolutions in the new variables as we will eventually estimate these terms.

The $n$-th Galerkin projection of~\eqref{eq:diag} is given by
\begin{equation}\label{eq:diagP}
  \begin{aligned}
    \dot u_k^+ &= \sqrt{ k^2 (\beta k^2 - 1 ) }u_k^+  + \frac{ \sigma k^2 N_{k,n}(u) + \epsilon f_k(t)}{2 \sqrt{ k^2 (\beta k^2 - 1 ) }} \,, \\
    \dot u_k^- &= -\sqrt{ k^2 (\beta k^2 - 1 ) }u_k^- - \frac{ \sigma k^2 N_{k,n}(u) + \epsilon f_k(t)}{2 \sqrt{ k^2 (\beta k^2 - 1 ) }},\ k = 1,\dots,n \,,
  \end{aligned}
\end{equation}
where
\begin{equation}
  N_{k,n}(u):=-2 \sum_{k_1 \geq  1}^{n-k} u_{k_1+k}u_{k_1} - \sum^{k-1}_{k_1 = 1} u_{k_1}u_{k-k_1}\,.
\end{equation}
Our next step is to construct a sequence of isolating segments $S^n$ for the Galerkin projections~\eqref{eq:diagP}.

\subsection{Construction of periodic isolating segments}

We look for periodic isolating segments $S^n$ of the form $S^n = [0,\tau] \times S_0^n$,
where
\begin{equation}\label{eq:s0n}
  S_0^n = \prod_{k=1}^{M} [u_k^l,u_k^r]^2 \oplus \prod_{k=M+1}^{n} [-C/k^s,C/k^s]^2 \,,
\end{equation}
i.e. the set of n-tuples of pairs $(u^+,u^-) = \{(u_k^-,u_k^+)\}_{k=1}^{n}$ 
such that $u_k^-, u_k^+ \in [u_k^l, u_k^r]$ for $k \in 1,\dots,M$ and $\left|u_k^\pm\right| \leq C/k^s$.
For now it is enough to take $C \in \rr^+$ and $s \in \{2,3,\dots\}$, however later on we will assume that $s$ is at least 6,
to comply with condition \textbf{C3} from the definition of self-consistent bounds.

Observe that we would like to choose the values of $C$ and $s$, as well as the first $M$ intervals the same for each projection.
Therefore we can say that our segments are a projection of an ``infinite-dimensional segment'' given by
\begin{equation}\label{eq:s0infty}
  S_0^\infty := \prod_{k=1}^{M} [u_k^l,u_k^r]^2 \oplus \prod_{k=M+1}^{\infty} [-C/k^s,C/k^s]^2 \,.
\end{equation}
The elements of $S_0^\infty$ are sequences of pairs $(u^+,u^-) = \{(u_k^-,u_k^+)\}_{k=1}^{\infty}$.
We denote them by the same symbols as elements of $S_0^n$ but we will always make it clear to element of which set we are referring to.

We would like to choose $u_k^l$, $u_k^r$, $C$ and $s$ such that the linear part of~\eqref{eq:diagP} dominates the nonlinear terms
and the isolation conditions~\eqref{eq:isolation} hold -- at least for sufficiently high modes, for $n$ large enough.
The inequalities for the low modes we will treat one-by-one with aid of rigorous numerics.

We assume the bounds for $u_k^+$ to be the same as the ones for $u_k^-$. As we will
see later, due to the symmetry of the equations~\eqref{eq:diagP}, the isolation conditions for both $u_k^+$ and $u_k^-$ are given by the same inequalities.

From estimates in \cite{ZgliczynskiMischaikow}\footnote{We note that the estimates
from \cite{Zgliczynski3} improve the bound on~\eqref{eq:nonest} to $\tilde{D}/k^{s}$ for some $\tilde{D}$,
but the one we use here is fine enough for our applications.} it follows that,
for a set $S_0^\infty$ of the form given above, there exists a constant $D \in \rr^+$
such that
\begin{equation}\label{eq:nonest}
  \sup\left\{ |N_{k,n}(u)|:\  n \geq k > M,\ (u^+,u^-) \in S_0^n \right\} < \frac{D}{k^{s-1}}\,.
\end{equation}
The value of $D$ can be given by an explicit formula, but we postpone its evaluation to Subsection~\ref{subsec:nonlinearity}.

\begin{lem}\label{lem:highmode}
  Assume that for some $M \in \nn^+$, we have $f_k=0, \ k > M$ and
  \begin{equation}\label{eq:minC}
    C > \frac{\sigma D}{M+1} \cdot \frac{1}{2 (\beta - (M+1)^{-2})}\,.
  \end{equation}
  Consider the Galerkin projection~\eqref{eq:diagP} for $n \geq M$ and let $S_0^n$ be given by~\eqref{eq:s0n}.
  Then, for $M<k \leq n$ and $(u^+,u^-) \in S_0^n$ the following inequalities hold
  \begin{align}
    \dot{u}_k^+ &> 0 \quad \text{if} \quad u_k^+ = C/k^s \,,\label{eq:i1}\\
    \dot{u}_k^+ &< 0 \quad \text{if} \quad u_k^+ = -C/k^s \,,\label{eq:i2}\\
    \dot{u}_k^- &< 0 \quad \text{if} \quad u_k^- = C/k^s \,,\label{eq:i3}\\
    \dot{u}_k^- &> 0 \quad \text{if} \quad u_k^- = -C/k^s \,.\label{eq:i4}
  \end{align}
\end{lem}

\begin{proof}
  We will prove~\eqref{eq:i1} and~\eqref{eq:i3}. The proof of~\eqref{eq:i2} and~\eqref{eq:i4}
  follows by reversing the inequality signs. We want 
  \begin{equation}
     \sqrt{ k^2 (\beta k^2 - 1 ) }u_k^\pm  + \frac{ \sigma k^2 N_{k,n}(u)}{2 \sqrt{ k^2 (\beta k^2 - 1 ) }} > 0\,.
  \end{equation}
  Since $u_k^\pm = C/k^s$, the above is equivalent to
  \begin{equation}
     \frac{C}{k^s} + \frac{ \sigma N_{k,n}(u)}{2 (\beta k^2 - 1 ) } > 0 \,.
  \end{equation}
  By the estimate~\eqref{eq:nonest} it is enough that
  \begin{equation}\label{eq:CIsol}
    C > \frac{ \sigma D k}{2 (\beta k^2 - 1 ) } = \frac{ \sigma D }{2 k (\beta - k^{-2} ) }
  \end{equation}
  and the right-hand side is at most $\frac{\sigma D}{M+1} \cdot \frac{1}{2 (\beta - (M+1)^{-2})}$.
\end{proof}

\subsection{Estimates for the nonlinear terms}\label{subsec:nonlinearity}

In this subsection we provide bounds for the nonlinear terms and compute $D$.
In fact we will look for an estimate
\begin{equation}
  N_{k,\max} < \frac{D}{k^{s-1}}, \quad k > M
\end{equation}
for
\begin{equation}
  N_{k,\max}:=\sup\{ |N_k(u)|: (u^+,u^-) \in S_0^\infty\}\,,
\end{equation}
as it is an upper bound on the left-hand side of~\eqref{eq:nonest} for all $n,k: n \geq k$.

Let $S_0^\infty$ be of the form as in~\eqref{eq:s0infty} and $(u^+,u^-) \in S_0^\infty$.
Recall that $u_k = u_k^+ + u_k^-$, hence
\begin{equation}
  \begin{aligned}
    u_k \in [ 2u_k^l, 2u_k^r ]\quad &\text{for} \quad k \leq M\,, \\
    |u_k| \leq \frac{2C}{k^s}\quad &\text{for}\quad k > M\,.
  \end{aligned}
\end{equation}

The nonlinearity consists of terms given by an infinite sum and a finite sum
\begin{equation}
  N_k(u) = -2IS(k)-FS(k)\,,
\end{equation}
where
\begin{align}
  IS(k)&:=\sum_{k_1 \geq  1} u_{k_1+k}u_{k_1} \,,\\
  FS(k)&:=\sum_{k_1=1}^{k-1}u_{k_1}u_{k-k_1} \,.
\end{align}
These terms, arising from the nonlinearity in the Kuramoto-Sivasinsky equation, were estimated in~\cite{ZgliczynskiMischaikow}
(cf. also \cite{Zgliczynski3}, Section 8).
Throughout the rest of this subsection we will denote by $u_k$ the whole interval $[ 2u_k^l, 2u_k^r ]$
and put $|u_k|:= 2 \max\{|u_k^l|, |u_k^r| \}$.

\begin{lem}[Lemma 3.1 in \cite{ZgliczynskiMischaikow}]\label{lem:31}
For $k \in \{1,\dots, M\}$ we have
\begin{equation}
  \begin{aligned}
    IS(k) \subset &\sum_{k_1=1}^{M-k} u_{k_1+k} u_{k_1} + 2C \sum_{k_1=M-k+1}^{M} \frac{|u_k|}{(k+k_1)^{s}}[-1,1]\\ &+ \frac{4C^2}{(k+M+1)^s (s-1)M^{s-1}}[-1,1]\,.
  \end{aligned}
\end{equation}
\end{lem}
\begin{lem}[Lemma 3.5 in \cite{ZgliczynskiMischaikow}]\label{lem:35}
For $k>2M$ we have
\begin{equation}
  FS(k) \subset \frac{2C}{k^{s-1}} \left( \frac{2^{s+1}}{2M+1} \sum_{k_1=1}^{M} |u_{k}| + \frac{C 2^{2s+1}}{(2M+1)^{s+1}} + \frac{C 2^{s+1}}{(s-1)M^s} \right)[-1,1]\,.
\end{equation}
\end{lem}

\begin{lem}[Lemma 3.6 in \cite{ZgliczynskiMischaikow}]\label{lem:36}
For $k > M$ we have
\begin{equation}
  IS(k) \subset \frac{2C}{k^{s-1}(M+1)}\left( \frac{2C}{(M+1)^{s-1}(s-1)} + \sum_{k=1}^{M} |u_{k}| \right)[-1,1]\,.
\end{equation}
\end{lem}

Following~\cite{Zgliczynski3} we give $D_{1}, D_{2}$ such that
\begin{equation}
  |FS(k)| \leq \frac{D_{1}}{k^{s-1}}, \quad |IS(k)| \leq \frac{D_{2}}{k^{s-1}}, \quad k>M
\end{equation}
and then set $D:=D_{1}+2D_{2}$.
Using Lemmas~\ref{lem:35} and~\ref{lem:36} we get the following formulas for $D_1$, $D_2$:
\begin{align}
  D_{1}(k\leq 2M) &= \max \{ k^{s-1} |FS(k)|,\ M < k \leq 2M \} \,, \\
  D_{1}(k > 2M) &= 2C \left( \frac{2^{s+1}}{2M+1} \sum_{k_1=1}^{M} |u_{k}| + \frac{C 2^{2s+1}}{(2M+1)^{s+1}} + \frac{C 2^{s+1}}{(s-1)M^s} \right) \,, \\
  D_{1} &= \max ( D_1 (k \leq 2M), D_1 (k > 2M) ) \,, \\
  D_{2} &=  \frac{2C}{M+1}\left( \frac{2C}{(M+1)^{s-1}(s-1)} + \sum_{k=1}^{M} |u_{k}| \right) \,.
\end{align}

\subsection{Low mode isolation and a procedure for refining the bounds}\label{subsec:refinement}

We will now discuss the low mode isolation inequalities. Assume,
that we found $M \in \nn$ and segments $S_0^n,\ n> M$ such that the assumptions of Lemma~\ref{lem:highmode}
hold. To apply Corollary~\ref{cor:isegment} to all Galerkin projections it is now enough to check
\begin{align}
  \dot{u}_k^+ &> 0 \quad \text{if} \quad u_k^+ = u_k^r \,,\label{eq:li1}\\
  \dot{u}_k^+ &< 0 \quad \text{if} \quad u_k^+ = u_k^l \,,\label{eq:li2}\\
  \dot{u}_k^- &< 0 \quad \text{if} \quad u_k^- = u_k^r \,,\label{eq:li3}\\
  \dot{u}_k^- &> 0 \quad \text{if} \quad u_k^- = u_k^l\label{eq:li4}\,.
\end{align}
for $(u^-,u^+) \in S_0^n$ and $n > M$.
It is enough to verify
\begin{equation}\label{eq:lowmodes}
  \begin{aligned}
    u_k^r  > -\frac{ \sigma k^2 N_{k}(u) + \epsilon f_k(t)}{2 k^2 (\beta k^2 - 1 ) }\,, \\
    u_k^l  < -\frac{ \sigma k^2 N_{k}(u) + \epsilon f_k(t)}{2 k^2 (\beta k^2 - 1 ) }\,.
  \end{aligned}
\end{equation}
for all $u \in S_0^\infty$, $t \in [0,\tau]$ and $k=1,\dots,M$.
Recall that $N_k(u) = -2 IS(k) - FS(k)$.
The term $IS(k)$ is bounded by use of Lemma~\ref{lem:31}, while $FS(k)$ are finite sums which can be for example rigorously enclosed
by use of interval arithmetics.
Therefore we can compute an explicit bound
\begin{equation}
  \sigma N_k(u) \subset [N_k^l, N_k^r], \quad u \in S_0^\infty 
\end{equation}
for each $k=1,\dots,M$. Assume that we are also have some bounds
\begin{equation}
  \epsilon f_k(t) \subset [f_k^l, f_k^r], \quad t \in \rr \,.
\end{equation}
Given these enclosures, inequalities~\eqref{eq:lowmodes} can be checked easily, in our case on a computer using interval arithmetics.
Note that there is no guarantee that for given $\epsilon \neq 0$ and given bounds $S_0^\infty$ the inequalities~\eqref{eq:lowmodes} will be satisfied.
However, for $|\epsilon|$ small enough ``good'' bounds should exist.
Since we cannot expect to choose the correct values for $\{u_k^{l,r}\}$, $C$ and $s$ at first try,
we will use an algorithm from~\cite{ZgliczynskiMischaikow} (Section 3.3) for refining an initial guess for the bounds.

Our goal is both to increase $s$ and correct our guesses for the bounds on coordinates where the isolation inequalities do not hold.
We iteratively adjust the pairs $(u_k^l,u_k^r)$, and later $C$ and $s$, so
at each step our new guess is at worst case an equality in the isolation conditions.
This way the bounds are tight on each coordinate, so the nonlinear terms do not contribute much error.
Note that the procedure is heuristic and we do not claim that the algorithm will produce correct bounds
-- this we verify a posteriori in interval arithmetics.

\begin{enumerate}
  \item First we adjust $C$ and $s$. Recall that, by~\eqref{eq:CIsol} for $k>M$ we want to choose $C$ and $s$
    such that
    \begin{equation}
      \frac{C}{k^s} > \frac{ \sigma D }{2 k^{s+1} (\beta - k^{-2} ) } \,.
    \end{equation}
    We want to increase $s$ - therefore, we set the new parameters by
    \begin{equation}
      \begin{aligned}
        s&:=s+1\,,\\
        C&:= \frac{\sigma D}{2 (\beta - (M+1)^{-2})}\,.
      \end{aligned}
    \end{equation}

  \item  Trying to comply with~\eqref{eq:lowmodes} we set the new $u_k^l$'s and $u_k^r$'s inductively for $k=1,\dots,M$ by
    \begin{equation}
      \begin{aligned}
        u_k^r&:= -\frac{ N_{k}^{l} +  k^{-2}f_k^l}{2(\beta k^2 - 1 ) } \,, \\
        u_k^l&:=-\frac{ N_{k}^{r} +  k^{-2}f_k^r}{2(\beta k^2 - 1 ) }\,.
      \end{aligned}
    \end{equation}
\end{enumerate}
After each run we check inequalities~\eqref{eq:lowmodes} and~\eqref{eq:minC} to see whether we obtained isolation.
We may however require additional iterates to improve $s$.

\subsection{Proof of Theorem~\ref{thm:main}}

For each of the parameter values and forcing terms we take a guess on the initial bounds with higher modes decay of order $\tilde{C}/k^4$.
Once a guess for some given range of $\epsilon$ is found it is easy to adapt it to another $\epsilon$ by rescaling proportionally
to the upper bound on $|\epsilon|$. After two iterates of procedure given in Subsection~\ref{subsec:refinement} we obtain
the values of $\{u_k^{l,r}\}$, $C$ and $s=6$ such that inequalities~\eqref{eq:lowmodes} and~\eqref{eq:minC} hold.
We present the approximate values (the first 5 significant digits of the actual values) in Table~\ref{tab3}.
We remark that in all of the cases it was enough to take $M=6$.
From Corollary~\eqref{cor:isegment} we conclude that for each Galerkin projection~\eqref{eq:diagP} for $n>M$ there exists a periodic solution $(u^{+,n}(t),u^{-,n}(t))$
such that
\begin{equation}\label{eq:upmBounds}
  u^{\pm,n}(t) \in \prod_{k=1}^{M} [u_k^l,u_k^r] \oplus \prod_{k=M+1}^{n} [-C/k^s,C/k^s], \quad \forall t \in \rr \,.
\end{equation}

The above computations were done on a computer in interval arithmetics, as it would be tedious
to do them by hand. The source files are available online~\cite{Czechowski}. The program
uses interval arithmetics implementation from the CAPD package~\cite{CAPD}.

\begin{table}
  \begin{center}
    \caption{Parameters $M$ and $s$ and approximate values of $C$, $u_k^{l}$ and $u_k^r$, $k=1,\dots,M$ used in the proof of Theorem~\ref{thm:main}}
\scalebox{0.73}{
    \begin{tabular}{ | c | c | c | c | c | c | c |}
        \hline $f \in$ & \multicolumn{3}{c}{$\mathcal{F}_{\tau}^{A}$} & \multicolumn{3}{|c|}{$\mathcal{F}_{\tau}^{B}$} \\ \hline
               $\beta$ & $1.5$ & $1.75$ & $2.5$ & $1.5$ & $1.75$ & $2.5$ \\ \hline
               $\epsilon$ & $[-0.05,0.05]$ & $[-0.1,0.1]$ & $[-0.3,0.3]$ & $[-0.05,0.05]$ & $[-0.1,0.1]$ & $[-0.3,0.3]$ \\ \hline 
               $M$ & $6$ & $6$ & $6$ & $6$ & $6$ & $6$ \\ \hline
        $u_1^r = -u_1^l$ & $0.05743$ & $0.082489$ & $0.16777$ & $0.059242$ & $0.085611$ & $0.17515$ \\  
$u_2^r = -u_2^l$ & $0.004018$ & $0.0069984$ & $0.020237$ & $0.0055667$ & $0.0097091$ & $0.026475$ \\  
$u_3^r = -u_3^l$ & $0.00022427$ & $0.0004798$ & $0.0019934$ & $0.00054628$ & $0.0010739$ & $0.0035043$ \\  
$u_4^r = -u_4^l$ & $1.1242 \times 10^{-5}$ & $2.9597 \times 10^{-5}$ & $0.00017727$ & $9.3307 \times 10^{-5}$ & $0.000179$ & $0.00055121$ \\  
$u_5^r = -u_5^l$ & $5.7862 \times 10^{-7}$ & $1.9415 \times 10^{-6}$ & $1.8646 \times 10^{-5}$ & $2.9174 \times 10^{-6}$ & $7.6705 \times 10^{-6}$ & $4.3434 \times 10^{-5}$ \\  
$u_6^r = -u_6^l$ & $5.5904 \times 10^{-7}$ & $1.9328 \times 10^{-6}$ & $2.1631 \times 10^{-5}$ & $7.7255 \times 10^{-7}$ & $2.596 \times 10^{-6}$ & $2.6004 \times 10^{-5}$ \\ \hline 
$C$ & $4.6941$ & $13.039$ & $100.64$ & $4.8878$ & $13.613$ & $102.99$ \\ \hline 

               $s$ & $6$ & $6$ & $6$ & $6$ & $6$ & $6$ \\ \hline
      \end{tabular} }
    \label{tab3}
  \end{center}
\end{table}

By the change of variables~\eqref{eq:changeVar}
we return now to the original coordinates $u_k,v_k$, i.e. the Fourier coefficients of $u$ and $u_t$ and obtain a sequence of periodic solutions
$(u^n(t),v^n(t))$ of the Galerkin projections of the system~\eqref{eq:infiniteOdeWF}.
From equation~\eqref{eq:upmBounds} and the form of our change of variables it follows, that there exists a $\hat{C} > 0$
(an exact value of which is not important to us), such that
\begin{equation}\label{eq:coeffuv}
  u^n_k(t) \leq \frac{\hat{C}}{k^6}, \qquad v^n_k(t) \leq \frac{\hat{C}}{k^{4}} \,.
\end{equation}
for all $n,k: k\leq n$ and $t \in \rr$.
A standard argument (cf. Theorem 10 in~\cite{Zgliczynski3}) proves that the set
\begin{equation}
  \prod_{k=1}^{\infty} [-\hat{C}/k^6, \hat{C}/k^6] \times [-\hat{C}/k^4, \hat{C}/k^4]
\end{equation}
satisfies conditions \textbf{C2}, \textbf{C3} and forms self-consistent bounds for~\eqref{eq:infiniteOdeWF}.
Note that at this moment we need the polynomial coefficient decay rate to be of order at least 2 for $v^n_k$'s and 6 for $u^n_k$'s (we have 4 and 6, respectively).
Let $u^n(t,x) = \sum_{k \in \nn} u_k(t) e^{ikx}$, $v^n(t,x) = \sum_{k \in \nn} v_k(t) e^{ikx}$.
From Theorem~\ref{thm:conv} it follows that the sequence $\{(u^n(t,x),v^n(t,x))\}_n$ has a subsequence $\{(u^{n_l}(t,x),v^{n_l}(t,x))\}_l$ 
converging uniformly on compact time intervals to a solution $(u^*(t,x),v^*(t,x))$ of
\begin{equation}\label{eq:firstorder}
  \begin{aligned}
    u_t &= v\,,\\
    v_t &= u_{xx} + \beta u_{xxxx} + \sigma (u^{2})_{xx} + \epsilon f(t,x)\,,
  \end{aligned}
\end{equation}
i.e. the Boussinesq equation~\eqref{eq:bsq} rewritten as a first order system. We have
\begin{equation}
  u^*(t,x) = \lim_{l \to \infty} u^{n_l}(t,x) = \lim_{l \to \infty} u^{n_l}(t+\tau,x) = u^*(t+\tau,x)
\end{equation}
for all $x,t \in \rr$, hence the solution is periodic.

The $C^0$ and $L^2$ bounds on $u^*$ and $v^*$ are computed from equations~\eqref{eq:changeVar} and~\eqref{eq:upmBounds}.
From the coefficient decay~\eqref{eq:coeffuv} and elementary facts about the Fourier series (see Section 6 in~\cite{ZgliczynskiNS}) it follows that $u^*(t,x)$ is
of class $C^4$ and $v^*(t,x)$ is of class $C^2$ as functions of $x$.

\begin{small}
\bibliographystyle{abbrv}
\bibliography{bsq_bib}
\end{small}


\end{document}